\documentclass[a4paper,10pt,reqno]{amsart}
\usepackage{amssymb,latexsym,bbm,amsmath,amsfonts}
\usepackage{amssymb,array}
\usepackage{mathtools}
\usepackage{graphicx}
\newcolumntype{L}{>{\centering\arraybackslash}m{1.5cm}}
\usepackage{amsthm, calrsfs}
\usepackage[mathscr]{eucal}
\usepackage{rotating}
\usepackage{multirow}
\usepackage{slashbox}
\usepackage{color}
\numberwithin{equation}{section}
\newtheorem{theorem}{Theorem}

\theoremstyle{definition}

\theoremstyle{remark}
\newtheorem{remark}{Remark}
\usepackage{enumerate}
\usepackage[colorlinks=true]{hyperref}

\newcommand{\RN}[1]{\uppercase\expandafter{\romannumeral#1}}

\begin{document}
\title[Subexponential Growth In FDEs]
{Subexponential Growth Rates in Functional Differential Equations}

\author{John A. D. Appleby}
\address{Edgeworth Centre for Financial Mathematics, School of Mathematical
Sciences, Dublin City University, Glasnevin, Dublin 9, Ireland}
\email{john.appleby@dcu.ie} \urladdr{webpages.dcu.ie/\textasciitilde
applebyj}

\author{Denis D. Patterson}
\address{School of Mathematical
Sciences, Dublin City University, Glasnevin, Dublin 9, Ireland}
\email{denis.patterson2@mail.dcu.ie} \urladdr{sites.google.com/a/mail.dcu.ie/denis-patterson}
\subjclass{Primary: 34K25; Secondary: 34C11.}
\thanks{Denis Patterson is supported by the Government of Ireland Postgraduate Scholarship Scheme operated by the Irish Research Council under the project GOIPG/2013/402.} 
\keywords{Functional differential equations, asymptotics, subexponential growth}
\date{\today}
\begin{abstract}
This paper determines the rate of growth to infinity of a scalar autonomous nonlinear functional differential equation with finite delay, where the right hand side is a positive continuous linear functional of $f(x)$. We assume $f$ grows sublinearly, and is such that solutions should exhibit growth faster than polynomial, but slower than exponential. Under some technical conditions on $f$, it is shown that the solution of the functional differential equation is asymptotic to that of an auxiliary autonomous ordinary differential equation with righthand side proportional to $f$ (with the constant of proportionality equal to the mass of the finite measure associated with the linear functional), provided $f$ grows more slowly than $l(x)=x/\log x$. This linear--logarithmic growth rate is also shown to be critical: if $f$ grows more rapidly than $l$, the ODE dominates the FDE; if $f$ is asymptotic to a constant multiple of $l$, the FDE and ODE grow at the same rate, modulo a constant non--unit factor. 
\end{abstract}
\maketitle

\section{Introduction}
In this paper, the growth rate to infinity of positive solutions of nonlinear autonomous functional differential equations of the form 
\begin{equation} \label{eq.fde}
x'(t) = \int_{[-\tau,0]}\mu(ds)f(x(t+s)), \,\, t >0,\,\, x_0 = \psi \in C([-\tau,0];(0,\infty)), 
\end{equation}
is studied. Here $\tau>0$ and $\mu$ is a positive finite Borel measure on $[-\tau,0]$ (so by definition $\mu(E)\in [0,\infty)$ for all Borel sets $E\subseteq[-\tau,0]$, and $\mu([-\tau,0])=:M\in (0,\infty)$). If $f$ is positive, by the Riesz representation theorem, \eqref{eq.fde} is equivalent to $x'(t)=L([f(x)]_t)$, $t>0$ where $L$ is a positive continuous linear functional from $C([-\tau,0];\mathbb{R}^+)$ to $\mathbb{R}^+$. Uniqueness of a continuous solution of \eqref{eq.fde} is guaranteed by asking that $f$ is continuously differentiable (see e.g. \cite{GLS} for existence results and properties of measures); positivity of solutions is guaranteed by the positivity of $\mu$ and of $f$ on $[0,\infty)$. Non--explosion of solutions in finite time, as well as subexponential growth to infinity of solutions (in the sense that $\log x(t)/t\to 0$ as $t\to\infty$) arises because $f'(x)\to 0$ as $x\to\infty$. Precise asymptotic results are obtained by asking that $f$ or $f'$ belong to the class of regularly varying functions (see \cite{BGT}). Recall that a measurable function $g:(0,\infty)\to (0,\infty)$ is regularly varying at infinity with index $\beta\in \mathbb{R}$ if $g(\lambda t)/g(t)\to \lambda^\beta$ as $t\to\infty$, for every $\lambda>0$. We write $g\in \text{RV}_\infty(\beta)$.

In the case when $f$ grows to infinity slightly slower than linearly (in the sense that $f\in\text{RV}_\infty(\beta)$ for $\beta<1$), it is known when $\mu(ds)=\delta_{\{0\}}(ds)+\lambda \delta{\{-\tau\}}(ds)$, that the rate of growth of solutions of \eqref{eq.fde} and of 
\begin{equation} \label{eq.equivode}
y'(t)=Mf(y(t)), \quad t>0; y(0)=y_0>0
\end{equation}
with $M= 1+\lambda$ is the same, in the sense that $x(t)/y(t)\to 1$ as $t\to\infty$ (see \cite{AMcR}). The non--delay equation \eqref{eq.equivode} can be considered as a special type of equation \eqref{eq.fde} in which all the mass $M$ of $\mu$ is concentrated at $0$. On the other hand, if $f$ is linear, collapsing the mass of $\mu$ to zero will grant different rates of (exponential) growth to solutions of \eqref{eq.fde} and \eqref{eq.equivode}. Therefore, the phenomenon that solutions of \eqref{eq.equivode} yield the growth rate of those of \eqref{eq.fde} ceases for some critical rate of growth of $f$ faster than functions in $\text{RV}_{\infty}(\beta)$ for $\beta<1$, but slower than linear. This suggests that the critical growth rate may be captured by a function $f$ in $\text{RV}_\infty(1)$ but with $f(x)/x\to 0$ (or $f'(x)\to 0$) as $x\to\infty$.

In our main result here (Theorem~\ref{thm.2.2}), we show that the critical rate of growth is $\text{O}(x/\log x)$: more precisely, if      
\begin{equation} \label{eq.lambda}
\lambda := \lim_{x \to \infty}\frac{f(x)}{x/\log(x)} \in [0,\infty],
\end{equation}
then $x(t)/y(t)\to \exp(-\lambda\int_{[-\tau,0]} |s|\mu(ds))$ as $t\to\infty$, provided $f$ is ultimately increasing and $f'\in \text{RV}_\infty(0)$, a hypothesis stronger than, but implying $f\in \text{RV}_\infty(1)$. In proving Theorem~\ref{thm.2.2}, we find that $F(x(t))/t\to M$ as $t\to\infty$ where 
\begin{equation} \label{def.F}
F(x) = \int_1^x \frac{1}{f(u)}du, \quad x>0
\end{equation}
and similarly $F(y(t))/t\to M$ as $t\to\infty$. Therefore our result identifies a subtle distinction in the growth rates of $x$ and $y$, which are in some sense close. 

Since \eqref{eq.fde} can be written, with $M=\int_{[-\tau,0]} \mu(ds)$, as 
\begin{equation} \label{delta}
x'(t)=Mf(x(t))-\int_{[-\tau,0]} \mu(ds) \{f(x(t))-f(x(t+s))\}=:Mf(x(t))-\delta(t),
\end{equation}
we can view \eqref{eq.fde} as a perturbation of \eqref{eq.equivode}, and if the perturbed term $\delta$ (which will be positive for large $t$, by the monotonicity of $x$ and $f$) is small relative to $Mf(x(t))$, we may expect $x(t)/y(t)$ to tend to a finite limit. This is in the spirit of 
a Hartman--Wintner type--result (see \cite[Cor X.16.4]{H}, \cite{HW}), so to gain insight into the asymptotic behaviour of \eqref{eq.fde}, we prove a nonlinear Hartman--Wintner theorem (Theorem~\ref{thm.2.1}), comparing the growth rate of the differential equations $x'(t)=Mf(x(t))-\epsilon(x(t))$ and $y'(t)=Mf(y(t))$, where $f(x)/x\to 0$ and $\epsilon(x)/f(x)\to 0$ as $x\to\infty$. Under an integral condition on $\epsilon$, we can show that $x(t)/y(t)$ tends to zero, unity or a non--trivial non--unit limit. Even though the result is for a simple scalar ODE, we were unable to find in the literature a result of this type. Furthermore, we believe this result is of independent interest, and can show, when allied with an analysis of the asymptotic behaviour of $\delta$, that Theorem~\ref{thm.2.1} identifies the critical growth rate of $f$ in \eqref{eq.lambda} for the FDE \eqref{eq.fde}, and predicts accurately that $x(t)/y(t)\to \exp(-\lambda\int_{[-\tau,0]} |s|\mu(ds))$ as $t\to\infty$.
For these reasons, the result is presented and proven here. We note of course, that there is a huge literature in asymptotic integration and Hartman--Wintner type--results in determining the asymptotic behaviour of nonlinear functional differential equations; some excellent, representative, papers include \cite{graef,KO,pituk}. Furthermore, the use of regular variation to analyse the asymptotic behaviour of ordinary differential equations is a very active field of research. Many threads to this research are presented in~\cite{Maric2000}.
   

\section{Results}
We start by proving a scalar, sublinear type of Hartman--Wintner theorem. 
\begin{theorem}\label{thm.2.1}
Suppose $f \in RV_\infty(1)$ is an increasing function, that $f(x)-\epsilon(x)>0$ for all $x>0$ and that the following limits hold as $x \to \infty$: 
\[
0 < \frac{\epsilon(x)}{f(x)} \to 0; \quad \frac{f(x)}{x} \to 0.
\]
If $f$ and $\epsilon$ are continuous, and $x$ and $y$ are the continuous solutions of 
\[
y'(t) = f(y(t)), \,\, t >0, \,\, y(0)>0; \quad x'(t) = f(x(t)) - \epsilon(x(t)), \,\, t >0, \,\, x(0)>0,
\]
and there is $\mu\in [0,\infty)$ such that 
\begin{align}\label{mu}
\lim_{x \to \infty}\frac{f(x)}{x} \int_0^x \frac{\epsilon(u)}{f^2(u)}du = \mu,
\end{align}
then 
\[
\lim_{t \to \infty}\frac{x(t)}{y(t)} = e^{-\mu}.
\]
\end{theorem}
\begin{proof}
The increasing, invertible functions 
\[
F(x) = \int_1^x \frac{1}{f(u)}du; \quad \Phi(x) = \int_1^x \frac{1}{f(u)-\epsilon(u)}du,
\]
are both well defined and we then have that
\[
F(y(t)) = F(y(0)) + t; \quad \Phi(x(t)) = \Phi(x(0))+t.
\]
Hence
\[
y(t) = F^{-1}(F(y(0)) + t); \quad x(t) = \Phi^{-1}(\Phi(x(0))+t).
\]
Since $(\epsilon(x) + f(x))/x \to 0$, $f(x)/x \to 0$ as $x \to \infty$, we have that 
$y(t) \sim F^{-1}(t)$ and $ x(t) \sim \Phi^{-1}(t)$ as $t \to \infty.$
Therefore it is sufficient to prove that $\Phi^{-1}(t) / F^{-1}(t) \to e^{-\mu}$ as $t \to \infty$. Define the function
\[
\Psi(x) = F(x) - \Phi(x) = - \int_1^x \frac{\epsilon(u)}{f^2(u) - f(u)\epsilon(u)}du.
\]
By hypothesis, we then have $\Psi(x) \sim - \mu [x/f(x)]$ as $x \to \infty$. Now let 
$
\Psi(x) = \left[ \bar{\epsilon}(x) - \mu \right]x/{f(x)},
$
where $\bar{\epsilon}(x) \to 0$ as $x \to \infty$. Thus, since $\Phi$ is invertible,
$
x = F(\Phi^{-1}(x)) - \Psi(\Phi^{-1}(x)).
$
For a fixed $x > 0$, $y = \Phi^{-1}(x)$ is the unique solution to $\delta_x(y) = 0$, where
$
\delta_x(y) = F(y) - \Psi(y) - x = F(y) - \left[\bar{\epsilon}(y) - \mu \right]y/{f(y)} - x. 
$
Suppose $K<1$ and let $z = K F^{-1}(x)$. Thus $x = F(z/K)$ and 
\begin{align*}
\delta_x(z) &= F(z) - \left[\bar{\epsilon}(z) - \mu \right]\frac{z}{f(z)} - F(z/K) 
= \int_{z/K}^z \frac{1}{f(u)}du - \left[\bar{\epsilon}(z) - \mu \right]\frac{z}{f(z)} \\
& = \frac{z}{f(z)}\left[ \mu - \bar{\epsilon}(z) + \frac{1}{z}\int_{z/K}^z \frac{f(z)}{f(u)}du \right] 
= \frac{z}{f(z)}\left[ \mu - \bar{\epsilon}(z) + \int_{1/K}^1 \frac{f(z)}{f(\alpha z)}d\alpha \right].
\end{align*}
Hence
\[
\frac{f(z)}{z}\delta_x(z) = \mu - \bar{\epsilon}(z) + \int_{1/K}^1 \frac{f(z)}{f(\alpha z)}d\alpha,
\]
and, since $\bar{\epsilon}(z) \to 0$, we have
\begin{align}\label{lim}
\lim_{x \to \infty}\frac{f(K F^{-1}(x))}{K F^{-1}(x)}\delta_x(K F^{-1}(x)) = \lim_{z\to \infty}\left(\int_{1/K}^1 \frac{f(z)}{f(\alpha z)}d\alpha \right) + \mu.
\end{align}
The function $\tilde{f}(x) := 1/f(x) \in RV_\infty(-1)$ and we then have
\begin{align*}
\int_{1/K}^1 \frac{f(z)}{f(\alpha z)}d\alpha &= \int_{1/K}^1 \frac{\tilde{f}(\alpha z)}{\tilde{f}(z)}d\alpha
= \int_{1/K}^1 \left( \frac{f(z)}{f(\alpha z)} - \alpha^{-1} \right) d\alpha + \int_{1/K}^1 \alpha^{-1} d\alpha.
\end{align*}
Applying the Uniform Convergence Theorem for Regularly Varying functions (Theorem 1.5.2 in \cite{BGT}) we return to \eqref{lim} to conclude that 
\begin{align*}
\lim_{x \to \infty}\frac{f(K F^{-1}(x))}{K F^{-1}(x)}\delta_x(K F^{-1}(x)) = \int_{1/K}^1 \frac{1}{\alpha}d\alpha + \mu = \log(K) + \mu.
\end{align*}
If $\mu >0$, then the above limit is positive for $K > e^{-\mu}$ and negative for $K < e^{-\mu}$. Now let $\epsilon \in (0,e^{\mu}-1) \cap (0,1)$
be arbitrary and consider
\[
\lim_{x \to \infty}\frac{f(e^{-\mu}(1-\epsilon) F^{-1}(x))}{e^{-\mu}(1-\epsilon) F^{-1}(x)}\delta_x(e^{-\mu}(1-\epsilon) F^{-1}(x)) 
= \log(1-\epsilon) < 0.
\]
Similarly, we obtain
\[
\lim_{x \to \infty}\frac{f(e^{-\mu}(1+\epsilon) F^{-1}(x))}{e^{-\mu}(1+\epsilon) F^{-1}(x)}\delta_x(e^{-\mu}(1+\epsilon) F^{-1}(x)) > 0.
\]
Therefore, there exist $x_1(\epsilon)$ and $x_2(\epsilon)$ such that for all $x \geq x^* := \max(x_1(\epsilon),x_2(\epsilon))$
\[
\delta_x(e^{-\mu}(1-\epsilon) F^{-1}(x)) < 0; \quad \delta_x(e^{-\mu}(1+\epsilon) F^{-1}(x)) > 0.
\] 
However, $\delta_x(y) = 0$ if and only if $y = \Phi^{-1}(x)$, and thus for all $x \geq x^*$ we have
\[
e^{-\mu}(1-\epsilon) F^{-1}(x) < \Phi^{-1}(x) < e^{-\mu}(1+\epsilon) F^{-1}(x).
\]
This allows us to conclude that 
\[
\lim_{x \to \infty}\frac{\Phi^{-1}(x)}{F^{-1}(x)} = e^{-\mu}, \,\, \mu \in (0,\infty).
\]
In the case when $\mu = 0$ we note that since $\epsilon(x)>0$ we have $F(x) < \Phi(x)$ and therefore $F^{-1}(x) > \Phi^{-1}(x)$.
Hence we may immediately conclude that
\[
\limsup_{t \to \infty}\frac{\Phi^{-1}(t)}{F^{-1}(t)} \leq 1.
\]
Recalling that $F(x) = \Phi(x) + \Psi(x)$ we have that $t = \Phi(F^{-1}(t)) + \Psi(F^{-1}(t))$. Thus
\[
F^{-1}(t) = \Phi^{-1}(t - \Psi(F^{-1}(t))) = u(t - \Psi(F^{-1}(t))),
\]
where $u(t) = \Phi^{-1}(t)$ and obeys $u'(t) = f(u(t))-\epsilon(u(t)), \,\, u(0)=1$. Next we write
\begin{align}\label{est}
\frac{\Phi^{-1}(t)}{F^{-1}(t)} = \frac{\Phi^{-1}(t)}{\Phi^{-1}(t- \Psi(F^{-1}(t)))} = \frac{u(t)}{u(t- \Psi(F^{-1}(t)))}.
\end{align}
Now by the Mean Value Theorem there exists $\theta_t \in [0,1]$ such that
$
u(t - \Psi(F^{-1}(t)))$ \\ $= u(t) - u'(t - \theta_t \Psi(F^{-1}(t)))\Psi(F^{-1}(t))$ \\$
= u(t) 
-\Psi(F^{-1}(t))\left[f(u(t - \theta_t \Psi(F^{-1}(t)))) - \epsilon(u(t - \theta_t \Psi(F^{-1}(t)))) \right].
$
\\Taking care to note that $\Psi(F^{-1}(t))\epsilon(u(t - \theta_t \Psi(F^{-1}(t))))<0$ we have the estimate
\begin{align*}
u(t - \Psi(F^{-1}(t))) &\leq u(t) - \Psi(F^{-1}(t))f(u(t- \Psi(F^{-1}(t)))) \\
&= \Phi^{-1}(t) - \Psi(F^{-1}(t))f(F^{-1}(t)).
\end{align*}
Putting this into \eqref{est} yields
\begin{align*}
\frac{\Phi^{-1}(t)}{F^{-1}(t)} &\geq \frac{\Phi^{-1}(t)}{\Phi^{-1}(t) - \Psi(F^{-1}(t))f(F^{-1}(t))}
= \frac{1}{1- \frac{\Psi(F^{-1}(t))f(F^{-1}(t))}{\Phi^{-1}(t)}}.
\end{align*}
Now let $\mu(t) = \Psi(F^{-1}(t))f(F^{-1}(t))/F^{-1}(t) < 0$. Thus
$
\mu(t)F^{-1}(t)/\Phi^{-1}(t) = \Psi(F^{-1}(t))f(F^{-1}(t))/\Phi^{-1}(t).
$
Hence
\[
\frac{\Phi^{-1}(t)}{F^{-1}(t)} \geq \frac{1}{1- \mu(t)\frac{F^{-1}(t)}{\Phi^{-1}(t)}}.
\]
Now multiply across by the strictly positive number $1- \mu(t) F^{-1}(t)/\Phi^{-1}(t)$ to obtain
$
\Phi^{-1}(t)/F^{-1}(t) \geq 1 + \mu(t).
$
By hypothesis, $\mu(t) \to 0$ as $t \to \infty$ and we have $\liminf_{t \to \infty}\Phi^{-1}(t)/F^{-1}(t) \geq 1$. 
Combining this with the limit superior gives the conclusion for $\mu = 0$.
\end{proof}
\begin{remark}
Theorem~\ref{thm.2.1} can be used to strongly motivate our main result and the argument by which is it proven. Consider the rearrangement 
\eqref{delta} of \eqref{eq.fde}    
where $M=\int_{-\tau}^{0}\mu(ds)$.
If, slightly modifying the hypotheses of Theorem~\ref{thm.2.1}, we let $f'\in RV_\infty(0)$, then it is shown (in the proof of Theorem~\ref{thm.2.2}) that $\delta(t) \sim M C f(x(t))f'(x(t))$ as $t \to \infty$, where $C:=\int_{-\tau}^0 |s|\mu(ds)$. With this in mind we revisit condition \eqref{mu} in Theorem~\ref{thm.2.1} and apply it to \eqref{delta} with $\epsilon(x) = MCf(x)f'(x)$, yielding
\begin{align*}
\lim_{x \to \infty}\frac{Mf(x)}{x}\int_0^x\frac{\epsilon(u)}{M^2 f^2(u)}du 
&=\lim_{x \to \infty} \frac{C f(x)}{x} \log(f(x)) =: \mu,
\end{align*}
provided $f(x)\log f(x)/x$ has a finite limit. Thus we are tempted to impose the condition that $f(x)\log(x)/x \to \lambda \in (0,\infty)$ as $x \to \infty$ and in this case we will have $\log(f(x)) \sim \log(x)$ as $x \to \infty$. Keeping faith in this analogy we are led to believe that, similar to Theorem~\ref{thm.2.1}, we should have
\[
\lim_{t \to \infty}\frac{x(t)}{F^{-1}(Mt)} = e^{-\mu} = e^{-\lambda C},
\]  
Our next result confirms that this intuition is in fact correct.
\end{remark}
\begin{theorem}  \label{thm.2.2}
Let $f(x)>0$ for all $x>0$, $f'(x)>0$ for all $x > x_1$, $f'(x) \to 0$ as $x \to \infty$, and $f' \in RV_\infty(0)$. If $\tau>0$, $f$ obeys \eqref{eq.lambda}, and $\mu\in M([-\tau,0];\mathbb{R}^+)$ is a positive finite Borel measure, the unique continuous solution $x$ of \eqref{eq.fde} obeys 
\begin{equation} \label{eq.xasy}
\lim_{t \to \infty}\frac{x(t)}{F^{-1}(Mt)} = e^{-\lambda C},
\end{equation}
where $F$ is given by \eqref{def.F}, $M := \int_{[-\tau,0]} \mu(ds)$ and $C := \int_{[-\tau,0]} |s|\mu(ds)$.
\end{theorem}
\begin{remark}
We note that under these hypotheses we have $f(x)/x \to 0$ as $x \to \infty$. Since $f$ is ultimately increasing it must either have a finite limit or tend to infinity as $x \to \infty$. In the former case, $x'(t)$ tends to a finite limit, and \eqref{eq.xasy} is trivially true.
\end{remark}
\begin{proof}
Our hypotheses on $\psi$ and the positivity of $f$ immediately yield that $x(t)\to\infty$ as $t\to\infty$. Thus there exists $T_1$ such that $x(t)> x_1$ for all $t \geq T_1$. Letting $t > T_1 + \tau$, and noting that $t \mapsto x(t)$ is increasing on $[0,\infty)$ we have 
\[
0 < x'(t) = \int_{[-\tau,0]}\mu(ds)f(x(t+s)) < \int_{[-\tau,0]}\mu(ds)f(x(t)) < M f(x(t)), \,\, t > T_1 + \tau.
\]
This means that $x'(t)/x(t) \to 0$ as $t \to \infty$. Furthermore, for $t > T_1 + \tau$, $f(x(t+s))>f(x(t-\tau))$ for $s \in [-\tau,0]$. Thus
$
x'(t) > M f(x(t-\tau)), \,\, t > T_1 + \tau.
$
Applying the Mean Value Theorem to the continuous function $f \circ x$ for each $t > T_1 + \tau$ there exists $\theta_t \in [0,\tau]$ such that 
$
f(x(t)) = f(x(t-\tau)) + f'(x(t-\theta_t))\tau.
$
Combining this identity with the fact that $f'(x) \to 0$ as $t \to \infty$, we see that $f(x(t-\tau))/f(x(t)) \to 1$ as $t \to \infty$. Hence 
$
\lim_{t \to \infty}x'(t)/f(x(t)) = M.
$
For each $t > T_1 + \tau$ and $s \in [-\tau,0]$ there is a $\theta_{t,s} \in [s,0] \subset [-\tau,0]$ such that  
\begin{align}\label{MVT_est}
0 < f(x(t)) - f(x(t+s)) &= (f \circ x)'(t- \theta_{t,s})s = f'(x(t- \theta_{t,s}))x'(t-\theta_{t,s})|s| \nonumber \\
&= (f \, f')(x(t-\theta_{t,s}))\frac{x'(t-\theta_{t,s})}{f(x(t-\theta_{t,s}))}|s|.
\end{align}
Now for every $\epsilon \in (0,1/2)$ there exists $T_2(\epsilon)>0$ such that 
\[
M(1-\epsilon) < \frac{x'(t)}{f(x(t))} < M, \mbox{ for all } t > T_2.
\]
$t > T_2 + \tau$ implies $t - \theta_{t,s}> T_2$ and hence
\begin{align}\label{est1}
M(1-\epsilon) < \frac{x'(t-\theta_{t,s})}{f(x(t-\theta_{t,s}))} < M, \mbox{ for all } t > T_2 + \tau.
\end{align}
Next $x(t-\tau)< x(t-\theta_{t,s})<x(t)$ for $t > T_1 + \tau$, $s \in [-\tau,0]$. Since $x'(t)/x(t) \to 0$ as $t \to \infty$, there is $T_3(\epsilon)$ such that $x(t-\tau)/x(t)> 1- \epsilon$ for all $t > T_3(\epsilon) + \tau$. Let $T_4 := T_1 + T_2 + T_3 + \tau$. Hence 
$
(1-\epsilon)x(t) < x(t-\theta_{t,s}) < x(t), \,\, s \in [-\tau,0], \,\, t > T_4.
$
Then with $\lambda_{t,s}:= x(t-\theta_{t,s})/x(t)$ we have $\lambda_{t,s} \in [1-\epsilon,1]$. Now let $t > T_4 + \epsilon$, so
\[
-1 + \frac{(ff')(x(t-\theta_{t,s}))}{(ff')(x(t))} = \frac{(ff')(\lambda_{t,s}x(t))}{(ff')(x(t))} - \lambda_{t,s} + \lambda_{t,s}-1.
\]
Hence for $s \in [-\tau,0]$, $t> T_4$, we have
\begin{align*}
\left|\frac{(ff')(x(t-\theta_{t,s}))}{(ff')(x(t))} -1 \right| &\leq \left|\frac{(ff')(\lambda_{t,s}x(t))}{(ff')(x(t))} - \lambda_{t,s} \right| + |\lambda_{t,s}-1| \\ &\leq \sup_{\lambda \in [1-\epsilon,1]}\left|\frac{(ff')(\lambda_{t,s}x(t))}{(ff')(x(t))} - \lambda_{t,s} \right| + \epsilon.
\end{align*}
Therefore 
\[
\lim_{x \to \infty}\sup_{\lambda \in [1/2-\epsilon,1]}\left|\frac{(ff')(\lambda_{t,s}x(t))}{(ff')(x(t))} - \lambda_{t,s} \right| =0.
\]
Hence, for every $\epsilon \in (0,1/2)$, there is $x_2(\epsilon)>0$ such that $x >x_2(\epsilon)$ implies
\[
\sup_{\lambda \in [1/2,1]}\left|\frac{(ff')(\lambda_{t,s}x(t))}{(ff')(x(t))} - \lambda_{t,s} \right| < \epsilon, \,\, x > x_2(\epsilon).
\]
Let $x(t) >x_2(\epsilon)$ for $t > T_5(\epsilon)$ and set $T_6 := T_4 + T_5$. Then for all $s \in [-\tau,0]$, $t \geq T_6$, 
\[
\left|\frac{(ff')(x(t-\theta_{t,s}))}{(ff')(x(t))} -1 \right| \leq 2\epsilon, \,\, t \geq T_6.
\]
Thus for $t > T_6$
\begin{align*}
f(x(t)) - f(x(t+s)) &= |s|\frac{(ff')(x(t-\theta_{t,s}))}{(ff')(x(t))} (ff')(x(t)) \frac{x'(t-\theta_{t,s})}{f(x(t-\theta_{t,s}))} \\
&= |s| \left(\frac{(ff')(x(t-\theta_{t,s}))}{(ff')(x(t))} -1 \right)(ff')(x(t)) \frac{x'(t-\theta_{t,s})}{f(x(t-\theta_{t,s}))}\\
&+ |s|(ff')(x(t)) \frac{x'(t-\theta_{t,s})}{f(x(t-\theta_{t,s}))}
:= T_1(s,t) + T_2(s,t). 
\end{align*}
Hence
$
\delta(t) := \int_{[-\tau,0]}\mu(ds)T_1(s,t) + \int_{[-\tau,0]}\mu(ds)T_2(s,t).
$
For $t \geq T_6$, we have \\
$
\left| \int_{[-\tau,0]} \mu(ds)T_1(s,t) \right| \leq 2\epsilon \int_{[-\tau,0]}|s|\mu(ds)M(ff')(x(t)),
$
and
\begin{align}\label{est2}
&\int_{[-\tau,0]} \mu(ds) T_2(s,t) \leq M \int_{[-\tau,0]} |s|\mu(ds)(ff')(x(t)),\\
&\int_{[-\tau,0]}\mu(ds) T_2(s,t) \geq M(1-\epsilon)\int_{[-\tau,0]} (ff')(x(t)).
\end{align}
Therefore we see that 
$
\lim_{t \to \infty}\delta(t)/M\int_{[-\tau,0]} |s|\mu(ds)(ff')(x(t)) = 1.
$
We note from \eqref{delta} that $x'(t)=Mf(x(t))-\delta(t)$.
Therefore, for $t \geq T_6$, our previous estimates yield
\begin{align*}
&x'(t) < M f(x(t)) - M C (1-\epsilon)(ff')(x(t)), \\
&x'(t) > M f(x(t)) - M C (1+\epsilon)(ff')(x(t)).
\end{align*}
Hence, with $x_3(\epsilon) = x(T_6(\epsilon))$, we have
\begin{align*}
\int_{x_3(\epsilon)}^{x(t)}\frac{du}{f(u) - C(1-\epsilon)f(u)f'(u)} &= \int_{T_6}^t \frac{x'(s)}{f(x(s))- C(1-\epsilon)f(x(s))f'(x(s))}ds \\ &\leq M(t-T_6), 
\end{align*}
and similarly
\begin{align*}
\int_{x_3(\epsilon)}^{x(t)}\frac{du}{f(u) - C(1+\epsilon)f(u)f'(u)} &= \int_{T_6}^t \frac{x'(s)}{f(x(s))- C(1+\epsilon)f(x(s))f'(x(s))}ds \\ &\geq M(t-T_6).
\end{align*}
For convenience of notation we define the functions
\begin{align*}
\Phi_{+\epsilon}(x) &:= \int_{x_3(\epsilon)}^{x(t)}\frac{du}{f(u) - C(1-\epsilon)f(u)f'(u)}, \\
\Phi_{-\epsilon}(x) &:= \int_{x_3(\epsilon)}^{x(t)}\frac{du}{f(u) - C(1+\epsilon)f(u)f'(u)}.
\end{align*}
Thus, for $t \geq T_6$, 
$
x(t) \leq \Phi_{+\epsilon}^{-1}(M(t-T_6)) $ and $
x(t) \geq \Phi_{+\epsilon}^{-1}(M(t-T_6)).
$
Define
$
y_{\pm \epsilon}'(t) = \phi_{\pm \epsilon}(y_{\pm \epsilon}(t)), \,\, t>0, \,\,
y_{\pm \epsilon}(0) = x_3,
$
where $\phi_{\pm \epsilon}(x) = f(x) - C(1 \pm \epsilon)f(x)f'(x)$. Then
\[
y_\epsilon(t) = \Phi_\epsilon^{-1}(t); \quad y_{-\epsilon}(t) = \Phi_{-\epsilon}^{-1}(t).
\]
Since $\phi_{\pm \epsilon}(x)/x \to 0$ as $x \to \infty$, it follows that $y_{\pm \epsilon}'(t)/y_{\pm \epsilon}(t) \to 0$ as $t \to \infty$. Thus for any $c \in \mathbb{R}$, $\lim_{t \to \infty}y_{\pm \epsilon}(t-c)/y_{\pm \epsilon}(t) = 1$,
and therefore 
\[
\lim_{t \to \infty}\frac{\Phi_{\pm \epsilon}^{-1}(Mt - M T_6)}{\Phi_{\pm \epsilon}^{-1}(Mt)} = 1.
\]
A short calculation reveals that
\begin{align*}
\Phi_\epsilon(x) = F(x) - F(x_3) + \int_{x_3}^x \frac{C(1-\epsilon)f'(u)}{f(u)\left[ 1-C(1-\epsilon)f'(u) \right]}du.
\end{align*}
Similarly
\[
\Phi_{-\epsilon}(x) = F(x) - F(x_3) + \int_{x_3}^x \frac{C(1+\epsilon)f'(u)}{f(u)\left[ 1-C(1+\epsilon)f'(u) \right]}du.
\]
Define
\[
\Psi_\pm(x) = \int_{x_3}^x\frac{C(1\pm \epsilon)f'(u)}{f(u)\left[ 1-C(1 \pm\epsilon)f'(u) \right]}du,
\]
so that $\Phi_{\pm \epsilon} = F(x) - F(x_3) + \Psi_\pm(x)$.
Next we apply L'H\^{o}pitals rule to compute
\begin{align*}
\lim_{x \to \infty}\frac{\Psi_\pm(x)}{C(1\pm \epsilon)\log(f(x))} = \lim_{x \to \infty}\frac{\int_{x_3}^x \frac{f'(u)}{f(u)}\frac{1}{1- C(1\pm\epsilon)f'(u)}du}{\log(f(x))} = 1.
\end{align*}
Now fix $x$ and let
$
\delta_+(y) := F(y) -F(x_3)+\Psi(y) - x.
$
We note that $\delta_+(\Phi_\epsilon^{-1}(x))=0$ and that $\delta_+$ is continuous. 
\begin{align*}
\delta_+(y) &= \frac{1}{f(y)} + \frac{C(1-\epsilon)f'(y)}{f(y)[C(1-\epsilon)f'(y)]}
= \frac{1}{f(y)[1-(1-\epsilon)f'(y)]} > 0.
\end{align*}
Therefore, $\delta_+(y)=0$ if and only if $y=\Phi_\epsilon^{-1}(x)$. Now with $z = K F^{-1}(x)$ calculate
\begin{align*}
\delta_+(z) &= F(z) - F(z/K) - F(x_3) + \Psi_+(z)
= \int_{z/K}^z \frac{du}{f(u)} - F(x_3) + \Psi_+(z). 
\end{align*}
Now $f' \in RV_\infty(0)$ implies that $f \in RV_\infty(1)$ and this gives us that 
\[
\int_{z/K}^z \frac{du}{f(u)} \sim \frac{z}{f(z)}\log(K) \mbox{ as } z \to \infty.
\]
Similarly, $\Psi_+(z) \sim C(1-\epsilon)\log(f(z)) \sim C(1-\epsilon)\log(z)$ as $z \to \infty$. Suppose that
$\lambda \in (0,\infty)$ in \eqref{eq.lambda}.
Therefore $\int_{z/K}^z du/f(u) \sim (1/\lambda) \log(z)\log(K)$. Thus
\begin{align}\label{asym}
\lim_{x \to \infty}\frac{\delta_\epsilon(K F^{-1}(x))}{\log(K F^{-1}(x))} = \frac{1}{\lambda}\log(K) + C(1-\epsilon).
\end{align}
Let $K_+(\epsilon)=(1+\epsilon)e^{-\lambda C(1-\epsilon)}$, $K_-(\epsilon)=(1-\epsilon)e^{-\lambda C(1-\epsilon)}$. This implies 
\begin{align*}
\lim_{x \to \infty}\frac{\delta_\epsilon(K_+(\epsilon) F^{-1}(x))}{\log(K_+(\epsilon) F^{-1}(x))} &= \log(1+\epsilon) >0, \\
\lim_{x \to \infty}\frac{\delta_\epsilon(K_-(\epsilon) F^{-1}(x))}{\log(K_-(\epsilon) F^{-1}(x))} &= \log(1-\epsilon) <0.
\end{align*}
Notice that $\lim_{\epsilon \to 0^+}K_\pm(\epsilon) = e^{-\lambda C}$. Thus for every $\eta \in (0,1)$ there is $x(\eta,\epsilon)>0$ such that $x > x(\eta,\epsilon)$ implies
\[
\frac{\delta_\epsilon(K_+(\epsilon) F^{-1}(x))}{\log(K_+(\epsilon) F^{-1}(x))} > (1-\eta)\log(1+\epsilon).
\]
Let $x_4(\epsilon) = x(1/2, \epsilon)$. Then $\delta_\epsilon(K_+(\epsilon)F^{-1}(x))>0$ for all $x > x_4(\epsilon)$. Similarly, for every $\eta \in (0,1)$, there is $x(\eta,\epsilon)>0$ such that $x> x(\eta,\epsilon)$ implies 
\[
\frac{\delta_\epsilon(K_-(\epsilon) F^{-1}(x))}{\log(K_-(\epsilon) F^{-1}(x))} < (1-\eta)\log(1-\epsilon).
\]
Let $x_5(\epsilon) = x(1/2,\epsilon)$. Thus $\delta_\epsilon(K_-(\epsilon)F^{-1}(x))<0$ for all $x > x_5(\epsilon)$. Let $x_6 = \max(x_4,x_5)$ and then
$
\delta_\epsilon(K_+(\epsilon)F^{-1}(x))>0 > \delta_\epsilon(K_-(\epsilon)F^{-1}(x)), \,\, x>x_6.
$
Thus $\Phi_\epsilon^{-1}(x) \in (K_-(\epsilon)F^{-1}(x),K_+(\epsilon)F^{-1}(x))$, $x>x_6$. Let $t>T_8$, $Mt > x_6$. Then
\begin{align}\label{limsup}
\frac{x(t)}{F^{-1}(Mt)} \leq \frac{\Phi_\epsilon^{-1}(M(t-T_8))}{\Phi_\epsilon^{-1}(Mt)}\frac{\Phi_\epsilon^{-1}(Mt)}{F^{-1}(Mt)} \leq \frac{\Phi_\epsilon^{-1}(M(t-T_8))}{\Phi_\epsilon^{-1}(Mt)}K_+(\epsilon).
\end{align}
Therefore 
$
\limsup_{t \to \infty}x(t)F^{-1}(Mt) \leq K_+(\epsilon),
$
and letting $\epsilon \to 0^+$ yields
\[
\limsup_{t \to \infty}\frac{x(t)}{F^{-1}(Mt)} \leq e^{-\lambda C}.
\]
Similarly we define
$
\delta_-(y) = F(y) - F(x_3) + \Psi_-(y) - x,
$
and $\Phi_{-\epsilon}^{-1}(x)$ is the unique solution to $\delta_-(y)=0$. An exactly analogous calculation to the above case yields 
\[
\frac{x(t)}{F^{-1}(Mt)} \geq \frac{\Phi_{-\epsilon}^{-1}(M(t-T_8))}{\Phi_{-\epsilon}^{-1}(Mt)}K_-(\epsilon).
\]
Therefore taking the liminf as $t \to \infty$ and letting $\epsilon \to 0^+$ we obtain
\[
\liminf_{t \to \infty}\frac{x(t)}{F^{-1}(Mt)} \geq e^{-\lambda C}.
\]
When $\lambda = + \infty$ the proof is almost identical up to \eqref{asym}, which becomes
\[
\lim_{x \to \infty}\frac{\delta_\epsilon(KF^{-1}(x))}{\log(KF^{-1}(x))} = (1-\epsilon) > 0.
\]
Letting $\epsilon = 1/4$, for all fixed $K<1$, and $x>x^*(K)$, we have $\delta_{1/4}(KF^{-1}(x))>0$ and hence $\Phi_{1/4}(x) < KF^{-1}(x)$ for all $x>x^*(K)$. Therfore, similarly to \eqref{limsup}, we obtain $\limsup_{t \to \infty}x(t)/F^{-1}(Mt) \leq K$, for any $K<1$. Sending $K \to 0^+$ and combining this with the trivial lower bound of zero on the liminf then yields
$\lim_{t \to \infty} x(t)/F^{-1}(Mt)= 0$. 
When $\lambda =0$ the argument necessarily differs slightly since the leading order asymptotics of $\delta_\epsilon(z)$ are now given by $z/f(z)$. Consequently, \eqref{asym} is replaced by
\[
\lim_{z \to \infty}\frac{\delta_\epsilon(z)}{z/f(z)} = \lim_{x \to \infty}\frac{\delta_\epsilon(KF^{-1}(x))}{KF^{-1}(x)/f(KF^{-1}(x))} = \log(K).
\]
Now define $K_+(\epsilon)$ and $K_-(\epsilon)$ and use the above limit to obtain
$
(1-\epsilon)F^{-1}(x) < \Phi_\epsilon^{-1}(x) < (1+\epsilon)F^{-1}(x), \,\, x>x^*(\epsilon).
$
Proceeding as in \eqref{limsup} and letting $\epsilon \to 0^+$ gives 
\[
\limsup_{t \to \infty}\frac{x(t)}{F^{-1}(Mt)} \leq 1.
\]
The argument for the liminf with $\lambda = 0$ can be obtained with the same modification to the argument in the $\lambda \in (0,\infty)$ case.
\end{proof}
\begin{remark}
Scrutinising the start of the proof, we see that if $f(x)>0$ for all $x\geq 0$, $f'(x)>0$ for all $x>x_1$ and $f'(x)\to 0$ as $x\to\infty$, 
we can show, \textit{without using the other hypotheses on $f$}, that $F(x(t))/t\to M$ as $t\to\infty$. This can be inferred directly from the limit $x'(t)/f(x(t))\to M$ as $t\to\infty$. The limit $F(x(t))/t\to M$ does not yield information on the behaviour of $x(t)/F^{-1}(Mt)$ because when $f\in \text{RV}_\infty(1)$, the function $F^{-1}$ is rapidly varying at infinity. Therefore, the rest of the proof of Theorem~\ref{thm.2.2} yields more refined information on the growth rate of solutions of \eqref{eq.fde}.  
\end{remark}
\begin{remark}
It is worthwhile to mention that with a slight modification of the above argument the hypothesis that $f'$ be regularly varying in this Theorem can be omitted entirely in the case when $\lambda =0$.
\end{remark}
\section{Example}
A simple example of an $f$ obeying the hypotheses of Theorem 2.2 is to take $f(x) = (x+1)/\log^\alpha(2+x)$, for $\alpha>0$. Clearly $f(x)>0$ for $x>0$ and 
\[
f'(x) = \frac{1}{\log^\alpha(2+x)}\left(1 - \frac{(1+x)\alpha}{(2+x)\log(2+x)} \right) > 0,\,\, x > e^\alpha-2.
\] 
It is easy to see that $f'(x) \to 0$ as $x \to \infty$ and that $f' \in RV_\infty(0)$. We also have that $\lambda$ in \eqref{eq.lambda} is 0, 1, or $+\infty$ according as to whether $\alpha$ is greater than, equal to, or less than, unity.
Making a substitution and splitting the resulting integral gives 
\begin{align*}
F(x) = \frac{1}{1+\alpha}\log^{\alpha+1}(x+2) - \frac{3^{\alpha+1}}{\alpha+1} + \int_3^y \frac{w^\alpha}{e^\alpha-1}dw.
\end{align*}
From here we readily derive that 
\[
F(x) \sim \frac{1}{1+\alpha}\log^{\alpha+1}(x), \,\, F^{-1}(x) \sim e^{(\alpha+1)^{\frac{1}{\alpha+1}}x^{\frac{1}{\alpha+1}}}, \mbox{ as } x \to \infty. 
\]
We note once more that $F^{-1}$ is rapidly varying at infinity so the rate of growth here is indeed subexponential but faster than any power function.
\section{Further Work}
In this short section, we suggest some further developments of the main result. The hypothesis \eqref{eq.lambda} is clearly crucial: and for subexponential growth, the condition that $f'(x)>0$ and tends to zero are natural and mild. However, granted these three hypotheses, one might expect to be able to relax the regular variation hypothesis on $f$, because in the case $\lambda\in (0,\infty]$, they imply $\log f(x)/\log x \to 1$ as $x\to\infty$, which is satisfied by any $f\in \text{RV}_\infty(1)$. 
It is also tempting to conjecture that Theorem~\ref{thm.2.2} also applies to analogous convolution Volterra equations where the measure $\mu$ is now supported on $[0,\infty)$. In this case, it is possible would be especially interesting in the light of 
\eqref{eq.xasy} to contrast the cases where $C=\int_{[0,\infty)} s\mu(ds)$ is finite and infinite.  
Finally, if we view Theorem~\ref{thm.2.2} as a Hartman--Wintner type--result, which yields exact asymptotic behaviour, it is also natural to ask if there are results which give less precise estimate of the rate of growth under weaker restrictions on $f$. In this direction, estimates of the form $F(x(t))/t\to M$ as $t\to\infty$, which are weaker than \eqref{eq.xasy}, are acceptable.   
We seek in a later work to investigate these three questions.


\begin{thebibliography}{99}

%

\bibitem{AMcR} 
\newblock J. A. D. Appleby, M. J. McCarthy, A. Rodkina,  
\newblock Growth rates of delay-differential equations and uniform
Euler schemes, 
\newblock in \emph{Difference equations and applications},  (eds. M. Bohner et al.), U\u{g}ur-Bah\c{c}e\c{s}ehir Univ. Publ.
Co., Istanbul, (2009), 117--124.

		
		\bibitem{BGT} 
\newblock N.~H.~Bingham, C.~M.~Goldie and J.~L.~Teugels, 
\newblock \emph{Regular Variation}, 
\newblock Encyclopedia of Mathematics and its Applications, 27, Cambridge University Press, Cambridge, 1987.


\bibitem{graef}  
\newblock J. R. Graef, 
\newblock Oscillation, nonoscillation, and growth of solutions of nonlinear functional differential equations of arbitrary order, 
\newblock \emph{J. Math. Anal. Appl.}, \textbf{60} (1977), 398–-409.

\bibitem{GLS} 
\newblock G. Gripenberg, S.-O. Londen and O. Staffans,
\newblock \emph{Volterra Integral and Functional Equations}, 
\newblock Encyclopedia of Mathematics, 34, Cambridge University Press, Cambridge, 1990.

\bibitem{HW} 
\newblock P. Hartman, A. Wintner, 
\newblock Asymptotic integration of ordinary nonlinear differential equations, 
\newblock \emph{Amer. J. Math.}, \textbf{77} (1955), 692--724.

\bibitem{H} 
\newblock P. Hartman, 
\newblock \emph{Ordinary differential equations}, 
\newblock 2$^{nd}$ edition, SIAM, Philadelphia, 2002.


\bibitem{KO} 
\newblock T. Kusano, H. Onose,
\newblock Oscillatory and asymptotic behavior of sublinear retarded differential equations
\newblock \emph{Hiroshima Math. J.}, \textbf{4} (1974), 343-–355.

\bibitem{Maric2000} 
\newblock V. Mari\'c, 
\newblock \emph{Regular Variation and Differential Equations}, 
\newblock Lecture Notes in Mathematics 1726, Springer-Verlag, Berlin, 2000.

\bibitem{pituk} 
\newblock M. Pituk, 
\newblock The Hartman--Wintner Theorem for Functional Differential Equations, 
\newblock \emph{J. Differential Equations}, \textbf{155} (1), (1999), 1--16. 

\end{thebibliography}
\end{document}